\documentclass[pdflatex,sn-mathphys-num]{sn-jnl}

\usepackage{graphicx}%
\usepackage{multirow}%
\usepackage{amsmath,amssymb,amsfonts}%
\usepackage{amsthm}%
\usepackage{mathrsfs}%
\usepackage[title]{appendix}%
\usepackage{xcolor}%
\usepackage{textcomp}%
\usepackage{manyfoot}%
\usepackage{booktabs}%
\usepackage{algorithm}%
\usepackage{algorithmicx}%
\usepackage{algpseudocode}%
\usepackage{listings}%

\theoremstyle{thmstyleone}%
\newtheorem{theorem}{Theorem}
\theoremstyle{thmstyletwo}%
\newtheorem{remark}{Remark}%

\theoremstyle{thmstylethree}%
\newtheorem{definition}{Definition}%

\begin{document}

\title[A Note on  Weak Post-Hopf Algebra Structures on the Sweedler Hopf Algebra]{A Note on  Weak Post-Hopf Algebra Structures on the Sweedler Hopf Algebra}


\author{\fnm{Q. G.} \sur{Chen}}\email{cqg211@163.com}

\affil{\orgdiv{College of Mathematics and Statistics}, \orgname{Kashi University}, \orgaddress{\street{Xuefu Avenue}, \city{Kaishi}, \postcode{844000}, \state{Xinjiang}, \country{China}}}


\abstract{By omitting the unitary constraint from the definition of weak post-Hopf algebras, we introduce the concept of relaxed weak post-Hopf algebras, offering a thorough characterization of all feasible relaxed weak post-Hopf algebraic structures on the Sweedler Hopf algebra. This work reveals a distinct class of relaxed weak post-Hopf algebraic configurations, diverging from previously established weak post-Hopf algebraic frameworks.}




\maketitle

\section{Introduction}\label{sec1}

The post-Lie algebra was initially introduced in [1],  boasting significant applications in geometric numerical integration [2,3]. Subsequently, in [5], Li, Sheng and Tang introduced the notion of post-Hopf algebra. In the formalization of post-Hopf algebras, the left multiplication $\alpha_{\rhd}: H\rightarrow$ End$(H)$ induced from $\rhd$ is convolution invertible. By removing the restriction that $\alpha_{\rhd}$ is convolution invertible, weak post-Hopf algebra was introduced in [4].

Typically, the unitarity condition of the actions is not invariably fulfilled. For example, consider the action of a Hopf algebra on an algebra through measurement and the concept of partial Hopf module actions, which emerged from the relaxation of the unitarity condition of actions in [6]. Naturally, a question arises: \textbf{if the unitarity condition is removed from the definition of weak post-Hopf algebras, would a distinct scenario unfold compared to that delineated in Theorem 3.1 of  [4]?}

This paper addresses this query by introducing the concept of relaxed weak post-Hopf algebras and providing a comprehensive delineation of all possible relaxed weak post-Hopf algebraic structures on the Sweedler Hopf algebra.

\section{Relaxed weak post-Hopf algebras}\label{sec2}

Throughout the paper, $\mathbf{k}$ shall denote an algebraically closed field of characteristic zero. All algebras and coalgebras under consideration are defined over $\mathbf{k}$ and the term “linear” is understood to mean $\mathbf{k}$-linear. Tensor products without explicit indication of the base field are assumed to be taken over $\mathbf{k}$. For an arbitrary coalgebra $(C, \Delta, \varepsilon)$,  we adopt the Sweedler-Heynemann notation for the comultiplication map $\Delta$, with the summation sign suppressed, as delineated below:
$$
\Delta(x)=x_{1}\otimes x_{2}.
$$

Let $H$ be a Hopf algebra. A group-like element of $H$ is a $g\in H$ which satisfies 
$$
\Delta(g)=g\otimes g,  \varepsilon(g)=1. 
$$
Denote by $G(H)$ the set of group-like elements in $H$, which is a group. For $g, h\in G(H)$, a $(g,h)$-primitive element is a $c\in H$ which satisfies 
$$
\Delta(c)=g\otimes c+c\otimes h. 
$$
Denote by $P_{g,h}(H)$ the subspace of $(g,h)$-primitive
elements in $H$ for $g, h\in G(H)$. Denote $P_{1,1}(H)$ by $P(H)$ the subspace of $(1,1)$-primitive elements ( primitive elements for short) in $H$, which is a Lie algebra.

 Now, we present the concept of a relaxed weak post-Hopf algebra by disregarding the unitary criterion of a weak post-Hopf algebra. 

Recall from [4] that a weak post-Hopf algebra is a pair $(H, \rhd)$, where
$H$ is a Hopf algebra and $\rhd: H\otimes H\rightarrow H$ is a coalgebra homomorphism satisfying the following equalities: 
\begin{equation}\label{eq1}
	x\rhd (yz)=(x_{1}\rhd y)(x_{2}\rhd z),
\end{equation}
\begin{equation}\label{eq2}
	x\rhd (y\rhd z)=(x_{1}(x_{2}\rhd y))\rhd z,
\end{equation}
\begin{equation}
	1\rhd x=x, 
\end{equation}
for any $x, y, z\in H$.

By ignoring the unitary condition, we introduce the following notion. 
\begin{definition}
	A relaxed weak  post-Hopf algebra is a pair $(H, \rhd)$, where 
	$H$ is a Hopf algebra and $\rhd: H\otimes H\rightarrow H$ is a coalgebra homomorphism satisfying the following equalities: 
	\begin{equation}\label{eq1}
		x\rhd (yz)=(x_{1}\rhd y)(x_{2}\rhd z),
	\end{equation}
	\begin{equation}\label{eq2}
		x\rhd (y\rhd z)=(x_{1}(x_{2}\rhd y))\rhd z,
	\end{equation}
	for any $x, y, z\in H$.
\end{definition}
 
 	Let $(H, \rhd)$ be a relaxed weak post-Hopf algebra. Then, for all $x\in H$, we have 
	\begin{equation}\label{eq9}
		x\rhd 1=\varepsilon(x)1. 
	\end{equation}

\section{The Main Theorem}
In this section, we classify all relaxed weak post-Hopf algebraic structures on the Sweedler 4-dimensional Hopf algebra. 

Recall that Sweedler's 4-dimensional Hopf algebra $\mathbb{H}_{4}$ is generated by two elements $g$ and $\nu$  which satisfy
$$
g^{2}=1, \nu^{2}=0, g\nu+\nu g=0.
$$
The comultiplication, the antipode and the counit of  $\mathbb{H}_{4}$ are given by
$$
\Delta(g)=g\otimes g, \Delta(\nu)=g\otimes\nu+\nu\otimes 1, \varepsilon(g)=1, \varepsilon(\nu)=0, S(g)=g, S(\nu)=-g\nu.
$$

\begin{theorem}\label{thm}
	Each relaxed weak post-Hopf structure on Sweedler's 4-dimensional Hopf algebra $\mathbb{H}_{4}$ has one form of the followings: 
	\begin{itemize}
		\item [(i)]$$\begin{tabular}{c|cccc}
			
			$\rhd $	& 1& g & $\nu$ & g$\nu$ \\
			\hline
			1	& 1 & g & $\nu$ & g$\nu$ \\
			
			g	& 1 & g & -$\nu$ &-g$\nu$  \\
			
			$\nu$	&  0& 0 & a$\nu$ &  ag$\nu$\\
			
			g$\nu$	& 0 &  0& a$\nu$ & ag$\nu$ \\
			
		\end{tabular}$$
		where $a$ is a parameter.
		\item [(ii)]
		$$\begin{tabular}{c|cccc}
			
			$\rhd $	& 1& g & $\nu$ & g$\nu$ \\
			\hline
			1	& 1 & g & $\nu$ & g$\nu$ \\
			
			g	& 1 & 1& 0 &0  \\
			
			$\nu$	&  0& a-ag & -a$\nu$ &  -ag$\nu$\\
			
			g$\nu$	& 0 &  a-ag& -a$\nu$ & -ag$\nu$ \\
			
		\end{tabular}$$
		where $a$ is a parameter.
		\item [(iii)]
		$$\begin{tabular}{c|cccc}
			
			$\rhd $	& 1& g & $\nu$ & g$\nu$ \\
			\hline
			1	& 1 & g & $\nu$ & g$\nu$ \\
			
			g	& 1 & g& $\nu$ & g$\nu$  \\
			
			$\nu$	&  0& 0 & 0 &  0\\
			
			g$\nu$	& 0 &0& 0 & 0\\
			
		\end{tabular}$$
		\item [(iv)]
	$$\begin{tabular}{c|cccc}
		
		$\rhd $	& 1& g & $\nu$ & g$\nu$ \\
		\hline
		1	& 1 & g & 0 &0 \\
		
		g	& 1 & g& 0 & 0  \\
		
		$\nu$	&  0& 0 & 0 &  0\\
		
		g$\nu$	& 0 &0& 0 & 0\\
		
	\end{tabular}$$
		\item [(v)]
	$$\begin{tabular}{c|cccc}
		
		$\rhd $	& 1& g & $\nu$ & g$\nu$ \\
		\hline
		1	& 1 & g & 0 &0 \\
		
		g	& 1 & 1& 0 & 0  \\
		
		$\nu$	&  0& 0 & 0 &  0\\
		
		g$\nu$	& 0 &0& 0 & 0\\
		
	\end{tabular}$$
		\item [(vi)]
	$$\begin{tabular}{c|cccc}
		
		$\rhd $	& 1& g & $\nu$ & g$\nu$ \\
		\hline
		1	& 1 & 1 & 0 &0 \\
		
		g	& 1 & 1& 0 & 0  \\
		
		$\nu$	&  0& 0 & 0 &  0\\
		
		g$\nu$	& 0 &0& 0 & 0\\
		
	\end{tabular}$$
	\end{itemize}
\end{theorem}

\begin{proof}
 Using the equation (\ref{eq9}), we can obtain 
$$
1\rhd 1=x,
g\rhd 1=1,  \nu\rhd 1=0=g\nu \rhd 1. 
$$

Notice easily that 
$$G(\mathbb{H}_{4})=\{1,g\},\, P_{1,1}(\mathbb{H}_{4})=\{0\},\, P_{g,1}(\mathbb{H}_{4})=\mathbf{k}\nu+\mathbf{k}(1-g), $$
$$P_{1,g}(\mathbb{H}_{4})=\mathbf{k}g\nu+\mathbf{k}(1-g),\, P_{g,g}(\mathbb{H}_{4})=\{0\}.$$

Since 
$$
\Delta(g\rhd g)=g\rhd g\otimes g\rhd g,
$$
\begin{eqnarray*}
	\Delta(g\rhd \nu )&=&(g\rhd \nu)\otimes (g\rhd 1)+(g\rhd g)\otimes (g\rhd \nu)\\
	&=&(g\rhd \nu)\otimes 1+(g\rhd g)\otimes (g\rhd \nu), 
\end{eqnarray*}
we have 
$$g\rhd g \in G(\mathbb{H}_{4}), g\rhd \nu \in P_{ g\rhd g,1}(\mathbb{H}_{4}),$$ 
then $g\rhd g=1$ or $g\rhd g=g$.

Let 
$$
1\rhd g=t_{1}1+ t_{2} g+ t_{3} \nu + t_{4} g \nu.
$$  
That $1 \rhd 1=1 \rhd (gg)=(1\rhd g)(1\rhd g)$ yields 

$$t_{1}^{2}+t_{2}^{2}=1, t_{1}t_{2}=t_{1}t_{3}=t_{1}t_{4}=0$$

\noindent \textbf{Case 1.}
If $t_{1}\neq 0$, then $t_{2}=t_{3}=t_{4}=0$, $t_{1}=1$ or $t_{1}=-1$.  We can exclude $t_{1}=-1$. The reasons are as follows:  If $t_{1}=-1$, then $1\rhd g=-1$, 
$$
\Delta(1\rhd g)=\Delta(-1)=-1\otimes 1,
$$
and 
$$
(\rhd\otimes \rhd) \Delta(1\otimes g)
=(\rhd\otimes \rhd) (1\otimes g\otimes 1\otimes g)=1\rhd g\otimes 1\rhd g=1\otimes 1,
$$
which shows that $\rhd$ is not a coalgebra homomorphism. 
Thus   $1\rhd g=1$.   
If $g\rhd g=g$, then
$$
g\rhd (1\rhd g)=g\rhd 1=1
$$
and 
$$
(g(g\rhd 1))\rhd g=g\rhd g=g,
$$
which does not satisfy the condition (2.2). 
Thus $g\rhd g=1$,then $g\rhd \nu \in P_{1,1}(\mathbb{H}_{4})$, i.e., $g\rhd \nu =0$. Thus 
\begin{eqnarray*}
	g\rhd (g\nu)=(g\rhd g)(g\rhd \nu)=0.
\end{eqnarray*} 
That $$
0=\nu \rhd 1=\nu \rhd (gg)=(\nu \rhd g)(1\rhd g)+(g\rhd g)(\nu\rhd g)=2(\nu\rhd g),
$$
follows 
$
 \nu\rhd g=0.
$
Since 
\begin{eqnarray*}
	\Delta(\nu\rhd \nu)&=&(\nu\rhd \nu)\otimes(1\rhd 1)+(g\rhd \nu)\otimes(\nu\rhd 1)\\
	&&+(\nu\rhd g)\otimes(1\rhd \nu)+(g\rhd g)\otimes(\nu\rhd \nu)\\
	&=&(\nu\rhd \nu)\otimes 1+1\otimes(\nu\rhd \nu), 
\end{eqnarray*}
we have $\nu\rhd \nu \in P_{1,1}(\mathbb{H}_{4})$, i.e., $\nu \rhd \nu =0$.
Thus 
$$
\nu\rhd (g\nu)=(g\rhd g)(\nu\rhd\nu)+(\nu\rhd g)(1\rhd\nu)=0.
$$
Since 
$$
\Delta(1\rhd \nu)=1\rhd g\otimes 1\rhd \nu + 1\rhd \nu \otimes 1 =1\otimes 1\rhd \nu + 1\rhd \nu \otimes 1,
$$
we have $1\rhd \nu \in P_{1,1}(\mathbb{H}_{4})$, i.e., $1 \rhd \nu =0$.
Similarly, we have $$1\rhd g\nu =0, g\nu\rhd \nu=0, g\nu\rhd g\nu=0.$$

\noindent \textbf{Case 2.} If  $t_{1}= 0$, then $t_{2}\neq 0$. Thus $1\rhd g= t_{2} g+ t_{3} \nu + t_{4} g \nu$. Using
$$
\Delta(1\rhd g)=(\rhd \otimes \rhd )\Delta(1\otimes g),
$$
we can obtain
$$
t_{2}=1, t_{3}=t_{4}=0.
$$
Thus
$
1\rhd g=g.
$
Since 
$$
\Delta(1\rhd \nu)=1\rhd g\otimes 1\rhd \nu + 1\rhd \nu \otimes 1 =g\otimes 1\rhd \nu + 1\rhd \nu \otimes 1,
$$
we have  $1\rhd \nu\in P_{g,1}(\mathbb{H}_{4})$. Let 
$$
1\rhd \nu=a\nu +b (1-g).
$$
That 
$$1\rhd (\nu\nu)=(1\rhd \nu)(1\rhd \nu)$$
yields $b=0$. 
Using 
$$
1 \rhd (1\rhd \nu)=(1(1\rhd 1))\rhd \nu=1\rhd \nu, 
$$
we have 
$$
a^{2}=a.
$$

\noindent \textbf{Case 2-1.}  If $a=0$, $1\rhd \nu =0$, and $1\rhd g\nu =(1\rhd g)(1\rhd \nu)=0$.

\noindent \textbf{Case 2-1-1.}  If $g\rhd g=1$, $g\rhd \nu =0$ (see case 1)  and \begin{eqnarray*}
	g\rhd (g\nu)=(g\rhd g)(g\rhd \nu)=0.
\end{eqnarray*} 
Since 
\begin{eqnarray*}
	\Delta(\nu\rhd g)&=&(\nu\rhd g)\otimes (1\rhd g)+(g\rhd g)\otimes (\nu\rhd g)\\
	&=&(\nu\rhd g)\otimes g+1\otimes (\nu\rhd g), 
\end{eqnarray*}
we have $\nu\rhd g\in P_{1,g}(\mathbb{H}_{4})$. Then there exist  $x, y\in \mathbf{k}$ such that 
$$
\nu\rhd g=x(1-g)+yg\nu.
$$
By using 
$$
0=\nu \rhd 1=\nu \rhd (gg)=(v\rhd g)(1\rhd g)+(g\rhd g)(\nu\rhd g)=(v\rhd g)g+\nu\rhd g,
$$
we have 
$$
-y\nu+yg\nu=0,
$$
which yields $y=0$. Thus $$
\nu\rhd g=x(1-g).$$
That 
$$
1\rhd (\nu \rhd g)=(1\rhd \nu )\rhd g=0
$$
yields 
$x=0$, i.e., $\nu\rhd g=0$. Thus $\nu \rhd \nu =0, \nu \rhd g\nu=0$. (see case 1). Similarly, we have 
$$
g\nu\rhd g=0, g\nu\rhd \nu =0, g\nu \rhd g\nu =0.
$$

\begin{eqnarray*}
0=\nu\rhd (g\rhd \nu)=(\nu(1\rhd g)+g(\nu\rhd g))\rhd \nu= -g\nu\rhd \nu
\end{eqnarray*}
\begin{eqnarray*}
p^{2}\nu	=\nu\rhd (\nu\rhd \nu)=(\nu(1\rhd \nu)+g(\nu\rhd \nu))\rhd \nu=p (g\nu\rhd \nu)=0
\end{eqnarray*}

\noindent \textbf{Case 2-1-2.}   If $g\rhd g=g$, then $g\rhd \nu \in P_{g,1}(\mathbb{H}_{4})$, i.e., there exist $c, d\in \mathbf{k}$ such that  
$$
g\rhd \nu =c\nu + d(1-g). 
$$
Since 
$$
g\rhd (g\rhd \nu)=(g(g\rhd g))\rhd \nu=g^{2}\rhd \nu=1\rhd \nu=0,
$$
one has
$$
c^{2}=0, (c+1)d=0, 
$$
i.e., $ c=d=0$.
So 
$
g\rhd \nu=0, $ moreover 
\begin{eqnarray*}
	g\rhd (g\nu)=(g\rhd g)(g\rhd \nu)=0.
\end{eqnarray*} 
Since 
$$
\Delta(\nu\rhd g)=(\nu\rhd g)\otimes (1\rhd g)+(g\rhd g)\otimes (\nu\rhd g)=(\nu\rhd g)\otimes g+g\otimes (\nu\rhd g),
$$
one has $\nu\rhd g \in  P_{g,g}(\mathbb{H}_{4})$ and thus $\nu\rhd g=0$. 

Since 
\begin{eqnarray*}
	\Delta(\nu\rhd \nu)&=&(\nu\rhd \nu)\otimes(1\rhd 1)+(g\rhd \nu)\otimes(\nu\rhd 1)\\
	&&+(\nu\rhd g)\otimes(1\rhd \nu)+(g\rhd g)\otimes(\nu\rhd \nu)\\
	&=&(\nu\rhd \nu)\otimes 1+g\otimes(\nu\rhd \nu),
\end{eqnarray*}
it follows that $\nu\rhd \nu\in P_{g,1}(\mathbb{H}_{4})$, i.e., there exist $p, q\in \mathbf{k}$ such that 
$$
\nu\rhd \nu =p\nu +q(1-g).
$$
By using 
$$
g\rhd (\nu \rhd \nu)=(g(g\rhd \nu))\rhd \nu=0,
$$
we have $
q=0
$
and $
\nu\rhd \nu =p\nu.
$
Since
\begin{eqnarray*}
0=\nu\rhd (g\rhd \nu)=(\nu(1\rhd g)+g(\nu\rhd g))\rhd \nu= -g\nu\rhd \nu
\end{eqnarray*}
and 
$$
p^{2}\nu=\nu\rhd (\nu\rhd \nu)=(\nu(1\rhd \nu)+g(\nu\rhd \nu))\rhd \nu=pg\nu\rhd \nu=0,
$$
we have 
$$
p=0,\nu \rhd \nu =0.
$$
Thus 
$$
\nu\rhd (g\nu)=(g\rhd g)(\nu\rhd\nu)+(\nu\rhd g)(1\rhd\nu)=0.
$$
Similarly, we have $$g\nu\rhd g=0, g\nu\rhd g\nu =0.$$

\noindent \textbf{Case 2-2.}  If $a=1$, then $1\rhd \nu =\nu$, $1\rhd g\nu=(1\rhd g)(1\rhd \nu)=g\nu$. We can repeat the proof of Theorem 3.1 in [4].
\end{proof}
\begin{remark}
	Theorem 1 gives us a surprise. After removing the unitality condition, the Sweedler Hopf algebra not only retains its original weak post-Hopf algebra structures, but also acquires two additional structures.
\end{remark}

{\bf Statement:} There are no conflict of interests!

\end{document}